\documentclass{article}[a4paper]

\title{Topological Dualities for Modal Algebras}
\author{
Matthew Collinson\\ 
\small{School of Natural and Computing Sciences},
University of Aberdeen,\\
\small{Aberdeen, AB24 3UE, United Kingdom,}
\texttt{\small{matthew.collinson@abdn.ac.uk}}
}

\date{} 

\usepackage{a4wide}
\usepackage{graphicx}
\usepackage[pdftex]{pict2e}
\usepackage{amsmath}
\usepackage{amsfonts}
\usepackage{amssymb}
\usepackage{amsthm}
\usepackage{makecell}
\usepackage{mathrsfs} 
\usepackage{bandd2025}

\newtheorem{theorem}{Theorem}
\newtheorem{lemma}{Lemma}
\newtheorem{proposition}{Proposition}
\newtheorem{corollary}{Corollary}

\theoremstyle{definition}
\newtheorem{definition}{Definition}

\newcommand{\Ccal}{\mathcal{C}} 
\newcommand{\Dcal}{\mathcal{D}} 
\newcommand{\Fcal}{\mathcal{F}} 
\newcommand{\Ocal}{\mathcal{O}} 
\newcommand{\Pcal}{\mathcal{P}} 
\newcommand{\Ofrak}{\mathfrak{O}} 

\newcommand{\topcat}{\mathbf{Top}}

\newcommand{\relsp}{\mathbf{RelSp}}
\newcommand{\relspl}{\mathbf{RelSp^l}}
\newcommand{\relspc}{\mathbf{RelSp^c}}

\newcommand{\relspq}{\mathbf{RelSpq}} 
\newcommand{\relspqu}{\mathbf{RelSpq^u}}
\newcommand{\relspql}{\mathbf{RelSpq^l}}
\newcommand{\relspqc}{\mathbf{RelSpq^c}}
\newcommand{\eqspq}{\mathbf{EqSpq}}
\newcommand{\eqspqc}{\mathbf{EqSpq^c}}

\newcommand{\frm}{\mathbf{Frm}}
\newcommand{\mfrm}{\mathbf{MFrm}} 
\newcommand{\mfrms}{\mathbf{MFrm}^{\scalebox{.5}{\algbox\!\!\!\algdiamond}}}
\newcommand{\mfrmbs}{\mathbf{MFrm}^{\scalebox{.5}{\algbox}}}
\newcommand{\mfrmds}{\mathbf{MFrm}^{\scalebox{.5}{\algdiamond}}}
\newcommand{\eqfrm}{\mathbf{EqFrm}}

\newcommand{\lmfrmds}{\mathbf{LMFrm}^{\scalebox{.5}{\algdiamond}}}
\newcommand{\cmfrms}{\mathbf{CMFrm}^{\scalebox{.5}{\algbox\!\!\!\algdiamond}}}
\newcommand{\ceqfrm}{\mathbf{CEqFrm}}

\newcommand{\mdlat}{\mathbf{MDLat}}

\newcommand{\twobb}{\mathbf{2}}
\DeclareMathOperator{\pt}{\mathrm{pt}}
\DeclareMathOperator{\charop}{\nabla}
\newcommand{\Pbb}{\mathbb{P}}

\newcommand{\invim}{^{\leftarrow}} 
\newcommand{\dirim}{^{\rightarrow\!}}
\newcommand{\interiorop}[1]{{#1}^\circ}
\newcommand{\hash}{{\#}}
\DeclareMathOperator{\downset}{\downarrow}
\DeclareMathOperator{\upset}{\uparrow}
\DeclareMathOperator{\compacts}{K}
\DeclareMathOperator{\ideals}{Idl}
\DeclareMathOperator{\Irs}{\mathscr{I}}

\begin{document}

\maketitle

\begin{abstract}
    We display a family of Stone-type dualities linking categories of frames carrying pairs of modal operators to categories of spaces carrying a binary relation. Different notions of morphism used on the relational side lead to significant variations in the point construction. We show how the situation simplifies in the case of semicontinuous relations, allowing for straightforward correspondences between modal axioms and relational properties. 
\end{abstract}

\section{Introduction}

Heyting algebras carrying unary modal operators arise in the study of intuitionistic propositional logics with modal connectives, generalizing the better known case in which Boolean algebras with operators arise from classical propositional logics. Classical propositional logic has a particularly useful semantics given on sets carrying a binary relation (Kripke frames) while intuitionistic propositional logic has a semantics given on topological spaces. Various authors have suggested giving semantics to intuitionistic modal logics by using topological spaces that carry a binary relation on points (`relational spaces'). Soundness and completeness of intuitionistic propositional calculus in the topological semantics can be understood in the framework of the Stone duality between coherent (spectral) spaces and distributive lattices, underpinned by the adjunction between spaces and frames (locales). In the classical case, Goldblatt showed that Stone duality extended to the modal case, specifically by requiring  semicontinuity of the binary relations on frames \cite{gold76a}. 

Hilken \cite{hil96},\cite{hil00} produced an adjunction in the Heyting case, a corresponding duality, and an existence of points result for a category of spectral frames.  The adjunction operates at an extremely general level, linking algebras satisfying a minimal set of axioms to relational spaces with no constraints on the binary relation carried. 
On the spatial side, a continuous version of the p-morphisms  of modal logic is used. 
However, the result is somewhat fragile. 
It is then difficult to specialize within it to produce results connecting categories of frames satisfying additional axioms to categories of relational spaces satisfying additional conditions. This means both ordinary enforcement of topological separation properties and, critical to logical applications,  modal correspondence theory, linking modal axioms to relational properties. 
Some (but not all) of the difficulty, relates to the point construction, where it is necessary to identify a set of points among a larger set of `pre-points'.


Parts of the territory have been explored by other authors. In most cases, the emphasis was on logic and models, and not on the relationship between categories of algebras and spaces. Wijesekera \cite{wijesekera1990constructive} is a significant precursor containing many parts of the same constructions.  
Several papers explore related ideas in lax logic \cite{FM97PLL,AMPR01}. 
These papers do not make explicit the morphisms on the spatial side. 
 A notion with some similarities to continuous p-morphisms is studied under the name `topo-bisimulation'  \cite{aiello2003reasoning}, but this is designed for interpretation of classical S4. 



In this paper we study a family of results related to Hilken's, in an attempt to provide a more uniform and useful account of the landscape. We give a refinement of the continuous p-morphism notion, which we call continuous pq-morphism, defined by an additional property, the q-morphism property, dual to the continuous p-morphism property. The q-morphism property interacts usefully with upper-semicontinuity of relations in passing from spaces to frames. With such morphisms, the notion of point of a modal frame then gives more control of the properties of the relational space constructed. 
Continuous pq-morphisms give an appropriate notion of functional bisimulation \cite{popkorn}, for which the intended interpretation of a modal propositional language (corrresponding to the modal algebras) is bisimulation invariant.

We show how the constructions simplify in the presence of semicontinuity properties of the relation. Semicontinuity properties can be enforced when a relational space is constructed by placing suitable restrictions on pre-points. Specifically, well-behaved relational spaces exist as subspaces. For modal frames supporting two additional axioms, all points are shown to satisfy these restrictions. Given the semicontinuity on the relational side, it is perhaps unsurprising that these axioms are equivalent, in the presence of the other axioms, to identities used in Johnstone's  construction of the Vietoris powerlocale \cite{joh82b}.  
These additional axioms and semicontinuity assumptions are not universally accepted as appropriate, either where the perspective relates to foundational concerns around constructive modal logics or to specific modelling concerns. Nevertheless, they are appropriate in many applications and have strong computational motivation \cite{abr91b,plo81,rob86a}.
Related restrictions are made by Goldblatt   \cite{gold76a} and by Wijesekera \cite{wijesekera1990constructive} within the topological semantics for additional axioms considered by Ewald \cite{ewald86} and Plotkin and Stirling \cite{plostir}.

With the semicontinuity properties and associated point restrictions, the modal correspondence theory becomes tractable. We show how it works in the case of equivalence relations. Equivalence relations are often used to support classical S5 modal logics and have also been used in the semantics of intuitionistic S5 logics \cite{simpson94}. At the same time, there has been recent interest in topological semantics for modal epistemic logics, where the topology is induced from observable properties of some state space \cite{baltagVanBenthem2025}, giving rise to relational spaces and to topologically valued transition systems \cite{collinson2026tvts}.

Correspondence theory in the more general cases without semicontinuity assumptions remains to be understood. There have been significant recent advances in technique that have been developed and applied to other related modal theories and semantics \cite{conradie2014unified,zhao2023correspondence}. A further challenge is to find more explicit  descriptions of the set of modal frame points as a subset of the set of pre-points. For controlling topological properties of relational spaces constructed, an alternative two-structure approach has been suggested \cite{collinson2026tvts}.

In Section~\ref{sec:adjunction} we introduce the categories of relational spaces of interest and the corresponding categories of modal algebras, and provide the adjunctions between pairs of these categories. In Section~\ref{sec:duality}, we show how the adjunctions restrict to dualities in the images of the functors participating in each of the adjunctions. Section~\ref{sec:existence} provides the `existence of points' results that show that classes of spectral/coherent modal frames lie within the scope of certain of these dualities. 

\section{Adjunction}
\label{sec:adjunction}

\subsection{Overview}

We introduce categories of spaces and categories of algebraic structures. We provide a family of adjunctions, with each member of the family matching a category of spaces with a category of algebraic structures. The functors in each direction are embellished forms of the functors involved in the adjunction between topological spaces and frames. The embellishments involve appropriate detail to handle the relational structure on spaces and modal structure on algebras. For relevant background, the reader may consult the books of Johnstone \cite{stonespaces} and Popkorn \cite{popkorn}.

\subsection{Spaces}

\begin{definition}
A 
\emph{relational space}  
$(X, \Ocal , R)$ consists of a topological space 
$(X, \Ocal)$ with carrier set $X$ and family of open sets $\Ocal$, 
together with a binary relation 
$R \subseteq X \times X$.
\end{definition}

We will write $R$ using infix notation, and also write $R\dirim(x) = \{ x' \in X \mid x R x' \}$ for $x\in X$. The relation $R$ and space $(X, \Ocal, R)$ are \emph{serial} if $R\dirim(x) \neq \emptyset$ for all $x \in X$. Definition~\ref{definition:classicalOperators} introduces the classical box and diamond operators of modal logic that operate on the powerset of the carrier set.

\begin{definition}
\label{definition:classicalOperators}
    Let $(X, \Ocal , \mathcal{R})$ be a relational space. For each $U \subseteq X$, define the sets
    \begin{align*}
        \algboxclass U &= \{ x \in X \mid R\dirim(x) \subseteq U \} \\
        \algdiamondclass U &= \{ x \in X \mid R\dirim(x) \cap U \neq \emptyset \} . 
    \end{align*}

    The relation $R$ \emph{upper-semicontinuous (u.s.c.)} with respect to $\Ocal$ if $\algboxclass U$ is open, for each open $U \in \Ocal$. The relation $R$ is \emph{lower-semicontinuous (l.s.c.)} if $\algdiamondclass U$ is open, for each open $U \in \Ocal$. The relation $R$ is \emph{continuous} if it is both upper semicontinuous and lower-semicontinuous. We will also say that the relational space is lower-semicontinuous, upper-semicontinuous or continuous, respectively, in each of these cases. If the relation $R$ is both continuous and an equivalence relation, then we say that the relational space is an \emph{equivalence space}.
\end{definition}

For a function $f : X \longrightarrow Y$, we write the inverse image of any subset $U \subseteq Y$ as $f\invim (U) = \{ x\in X \mid f(x) \in U \}$. 

\begin{definition} 
Let 
$f:(X,\Ocal, R ) \rightarrow (Y, \Pcal , S )$ be a continuous function between relational spaces. 
Consider the following three conditions:
\begin{eqnarray}
\label{relmorph}
x R y & \Rightarrow & f(x) S f(y)  \\
\label{pmorph} 
 f(x) S y \in U & \Rightarrow &  
            \exists x' \in X. ~ x R x' ~ \& ~ x' \in f \invim (U) \\
\label{qmorph}
f(x)S y \notin U  & \Rightarrow &  
                  \exists x' \in X. \ xR x' \ \& \ x' \notin f \invim (U) .
\end{eqnarray}

The function $f$ is a \emph{continuous relational morphism} if it satisfies (\ref{relmorph}) for all $x,y \in X$.
The function $f$ is a \emph{continuous p-morphism} if it satisfies (\ref{relmorph}) and (\ref{pmorph}) for all $U \in \Pcal$, $x\in X$ and $y \in Y$.  
The function is $f$ a \emph{(continuous) pq-morphism} if it is a continuous p-morphism and it also satisfies (\ref{qmorph}) for all $U \in \Pcal$, $x \in X$ and $y \in Y$. 

The names and definitions of various categories of relational spaces that we use in this paper are collated in Table~\ref{table:spaces}. 
\begin{table}[ht]
\begin{center}
\begin{tabular}{|l|l|l|}
\hline
Name    & Objects  & Morphisms \\ 
\hline
$\relsp$   & relational spaces & continuous p-morphisms \\
$\relspl$   & l.s.c. relational spaces & continuous p-morphisms\\
$\relspc$   & continuous relational spaces & continuous p-morphisms\\
$\relspq$  & relational spaces & continuous pq-morphisms \\ 
$\relspqu$ & u.s.c. relational spaces & continuous pq-morphisms \\ 
$\relspql$ & l.s.c. relational spaces & continuous pq-morphisms \\
$\relspqc$ & continuous relational spaces & continuous pq-morphisms \\
$\eqspq$  & equivalence spaces & continuous pq-morphisms \\
\hline
\end{tabular}
\caption{Categories of relational spaces}
\label{table:spaces}
\end{center}
\end{table}
\end{definition}

\subsection{From Spaces to Frames}

\begin{definition}
A \emph{modal frame} 
$(A , \algbox , \algdiamond )$ 
consists of a frame 
$A = (A, \leq , \bigvee , \wedge , \top , \bot)$ together with unary operators $\algbox:A \longrightarrow A$ and $\algdiamond: A  \longrightarrow A$ that are monotonic with respect to $\leq$ and that satisfy the following, for all $a,b \in A$:
\begin{eqnarray}
\label{m1} 
    \top  & \leq & \algbox \top      \\
\label{m2} 
    \algbox a \wedge \algbox b  & \leq & \algbox (a \wedge b)  \\
\label{m3}
    \algbox a \wedge \algdiamond b & \leq & 
    \algdiamond (a \wedge b)  \\
\label{m4}
    \algdiamond \bot & \leq    & \bot                  .
\end{eqnarray}

A modal frame is \emph{lower} if it additionally satisfies 
\begin{equation}
\label{axiom:diamondpreservesbinaryjoin}
\algdiamond (a \vee b) \leq 
\algdiamond a \vee \algdiamond b  
\end{equation}
for all $a, b \in A$.
%
It is \emph{convex} if it is lower and the following axiom holds
\begin{align}
     \label{axiom:lscboxveemixed}
    \algbox (a \vee b) \leq \algbox a \vee \algdiamond b
    %
\end{align}
for all $a,b \in A$.

A modal frame $A$ is \emph{serial} if 
\begin{align}
    \label{ax:serial}
    \algbox a \leq \algdiamond a 
\end{align}
for all $a \in A$.
\end{definition}

The axioms (\ref{m1}),(\ref{m2}),(\ref{m3}),(\ref{m4}),(\ref{axiom:diamondpreservesbinaryjoin}),(\ref{axiom:lscboxveemixed}) are equivalent to the part of the identities defining the Vietoris powerlocale that use only finitary operations \cite{joh82b}. 
Axiom (\ref{axiom:lscboxveemixed}) has been used in other works on intuitionism and modal logic \cite{ewald86,plostir}.

\begin{definition}
    An \emph{equivalence frame} is a modal frame $(A , \algbox , \algdiamond )$ that additionally satisfies all of the following conditions for all $a \in A$:
    \begin{align}
        \label{reflexivitybox}
        \algbox a 
        & \leq 
        a\\
        \label{reflexivitydiamond}
        a   
        & \leq\algdiamond 
        a\\
        \label{transitivitybox}
        \algbox a 
        & \leq 
        \algbox \algbox a\\
        \label{transitivitydiamond}
        \algdiamond \algdiamond a 
        & \leq 
        \algdiamond a \\
        \label{symmetrydiamondbox}
        \algdiamond\algbox a 
        & \leq 
        a\\
        \label{symmetryboxdiamond}
        a 
        & \leq 
        \algbox\algdiamond a .
    \end{align}
\end{definition}
A system of axioms in logical form equivalent to 
(\ref{ax:serial}), 
(\ref{reflexivitybox}), 
(\ref{reflexivitydiamond}), 
(\ref{transitivitybox}), 
(\ref{transitivitydiamond}), 
(\ref{symmetrydiamondbox}), 
(\ref{symmetryboxdiamond}) appear in Simpson's thesis \cite{simpson94} in studying a version of intuitionistic S5 modal logic.

\begin{definition}
A \emph{modal frame morphism } 
$f:(A , \algbox , \algdiamond ) \longrightarrow (B ,  \algbox , \algdiamond )$ 
is a morphism of frames $f:A \longrightarrow B$, such that
\begin{align}
\label{boxmorphcond}
    f \circ \algbox & \leq \algbox \circ f \\
\label{dimorphcond} 
    f \circ \algdiamond & \leq \algdiamond \circ f . 
\end{align}
Here, the relation $\leq$ is the pointwise extension of the order 
on frames to an order on frame morphisms.
The morphism is \emph{box-strict} if (\ref{boxmorphcond}) is an equality, \emph{diamond-strict} if (\ref{dimorphcond}) is an equality and  
\emph{strict}  
if it is both box-strict and diamond-strict.

The names of various categories of frames with modal operators are displayed in Table~\ref{table:CatsOfFrames}.
\begin{table}[ht]
\begin{center}
\begin{tabular}{|l|l|l|}
\hline
Name    & Objects & Morphisms  \\ 
\hline
$\mfrm$ & modal frames & modal frame morphisms\\
$\mfrmbs$ & modal frames & box-strict modal frame morphisms\\
$\mfrmds$ & modal frames & diamond-strict modal frame morphisms\\
$\lmfrmds$ & lower modal frames & diamond-strict modal frame morphisms\\
$\mfrms$ & modal frames & strict modal frame morphisms \\
$\cmfrms$ & convex modal frames & strict modal frame morphisms\\
$\eqfrm$ & equivalence frames & strict modal frame morphisms\\
$\ceqfrm$ & convex equivalence frames & strict modal frame morphisms\\
\hline
\end{tabular}
\caption{Categories of Frames with Operators}
\label{table:CatsOfFrames}
\end{center}
\end{table}
\end{definition}

\begin{definition}
\label{def:OmegaAssignments}
    Given a relational space $(X , \Ocal , R )$, define 
    \begin{equation*}
        \Omega X = (\Ocal , \algbox , \algdiamond )
    \end{equation*} 
    to consist of the frame $\Ocal$ of open sets  together with the operations defined by 
    \begin{align}
    \label{intBoxInterp}
        \algbox U 
        &= 
        \interiorop{(\algboxclass U)} 
        \\
        \algdiamond U 
        & = 
        \interiorop{(\algdiamondclass U)}  
        \label{intDiaInterp} 
    \end{align}
    for all $U \in \Omega X$, where $\interiorop{(-)}$ denotes the topological interior operation.

    For a continuous function 
    $f:(X, \Ocal , R) \longrightarrow (Y, \Pcal , S)$, define 
    \begin{equation*}
        \Omega f = f\invim : \Pcal \longrightarrow \Ocal .
    \end{equation*} 
\end{definition}

Lemma~\ref{results:Omega} gathers together some important properties of modal frames and their morphisms arising from relational spaces. All parts are straightforward and we omit the proofs.

\begin{lemma}
\label{results:Omega}
Let $(X, \Ocal , R)$ be a relational space and $f:(X, \Ocal , R) \longrightarrow (Y, \Pcal , S)$ be a continuous function between relational spaces. The following then hold: 
\begin{enumerate}
\item $\Omega X$ is a modal frame. 


\item If $(X, \Ocal , R)$ is lower semicontinuous, then $\Omega X$ is a lower modal frame.

\item If $(X, \Ocal , R)$ is continuous, then $\Omega X$ is a convex frame.

\item    
If $(X, \Ocal , R)$ is an equivalence space, then $\Omega X$ is an equivalence frame. 

\item If $(X, \Ocal , R)$ is serial, then $\Omega X$ is serial.

\item If $f$ 
is a continuous relational morphism, then 
$\Omega f$ satisfies (\ref{boxmorphcond}). If $f$ satisfies the continuous p-morphism condition (\ref{pmorph}), then $\Omega f$ satisfies (\ref{dimorphcond}). If $f$ is a continuous p-morphism, then $\Omega f$ is a morphism of modal frames.

\item If $f$ is a continuous relational morphism and $(Y,\Pcal,S)$ is lower-semicontinuous, then $\Omega f$ is diamond-strict.

\item If $f$ satisfies the q-morphism property (\ref{qmorph}) and $(Y,\Pcal,S)$ is upper-semicontinuous, then $\Omega f$ is box-strict.
\end{enumerate}
\end{lemma}

In the case of an equivalence frame arising from an equivalence space under the action of the functor $\Omega$, the two reflexivity axioms (\ref{reflexivitybox}) and  (\ref{reflexivitydiamond}) and one of the transitivity axioms (\ref{transitivitybox}) do not depend on continuity properties of the relation. However, the other equivalence frame axioms require topological assumptions 

\begin{lemma}
The assignments under $\Omega$ define functors with source and target categories as listed in Table~\ref{table:OmegaFunctorsNEW}. 
\begin{table}[ht]
\begin{center}
\begin{tabular}{|l|l|}
\hline
Source    & Target  \\ 
\hline
$\relsp$   & $\mfrm$ \\
$\relspl$  & $\lmfrmds$ \\
$\relspq$  & $\mfrm$ \\
$\relspql$ & $\lmfrmds$ \\
$\relspqu$ & $\mfrmbs$ \\ 
$\relspqc$ & $\cmfrms$ \\
$\eqspq$  & $\ceqfrm$ \\
\hline
\end{tabular}
\caption{Functors defined under assignments $\Omega$}
\label{table:OmegaFunctorsNEW}
\end{center}
\end{table}
\end{lemma}

Suppose that a set $A_0$ generates a modal Heyting algebra, in particular so that the system of inequalities (\ref{m1}), (\ref{m2}), (\ref{m3}), (\ref{m4}) are restated as a system of equalities. Suppose the existence of a continuous function 
$f: (X , \Ocal , R) \longrightarrow (Y, \Pcal , S)$ 
between relational spaces, and that there are valuation functions 
$v_\mathcal{I}: A_0 \longrightarrow \Ocal$ and 
$v_ \mathcal{J} : A_0 \longrightarrow \Pcal$. These can be extended homomorphically to interpretation functions 
$\mathcal{I}: A \longrightarrow \Ocal$ and 
$\mathcal{J} : A \longrightarrow \Pcal$, 
in particular using (\ref{intBoxInterp}) and (\ref{intDiaInterp}). 
Suppose also that 
$f \invim \circ v_\mathcal{J} = v_\mathcal{I}$. 
If $f$ is a continuous pq-morphism and an open map, then 
$f \invim \circ \mathcal{J} = \mathcal{I}$.
In modal logician's terms, an open and continuous pq-mporhism between models is a functional bisimulation: using the standard notation for logical satisfaction at a point  
\begin{equation*}
    x \vDash _{\mathcal{I}} \phi \Longleftrightarrow fx \vDash _\mathcal{J} \phi
\end{equation*}
for all $x \in X$ and $\phi$ in the modal language.



\subsection{Constructing Spaces}

Recall the construction of a topological space in the category $\topcat$ from a frame in the category $\frm$ of frames \cite{stonespaces}. For a frame $A$, the space constructed is $\pt{(A)}$. The points of the space can be presented as \emph{frame characters}, that is, as frame morphisms from $A$ to the two-element frame $\twobb$ with carrier set $\{0=\bot , 1 = \top \}$. Frame characters $p:A \longrightarrow \twobb$ are in one-one correspondence with completely prime filters of $A$, and we let 
\begin{equation*}
    \charop(p) = \{ a \in A \mid p(a)=1 \} .
\end{equation*}

We now show how to construct relational spaces from various kinds of modal frames. Different sets of points will be required in different versions of the constructions. 

\subsubsection{Pre-points}

Let $A$ be a modal frame. For any frame character $p:A \longrightarrow \twobb$, define 
\begin{align}
    \label{canonicalFilter}
    F_p & = \{ c \in A \mid p (\algbox c)=1\} \\ 
    \label{canonicalElement}
     a_p & = \bigvee \{ c \in A \mid p(\algdiamond c)=0 \} .
\end{align}

Lemma~\ref{result:FpFilter} is straightforward, by monotonicity of the box operator, and by axioms (\ref{m2}) for modal frames and (\ref{axiom:lscboxveemixed}) for convexity.

\begin{lemma}
\label{result:FpFilter}
If $A$ is a modal frame and $p$ is any frame charater, then the set $F_p$ is a filter. If $A$ is convex, $c \notin F_p$ and $p(\algdiamond b) = 0$, then $b \vee c \notin F_p$.
\end{lemma}

For any subset $S$ of a partial order $A$, let 
$\downset S = \{ b \in A \mid \exists a \in S. \ b \leq a \}$ 
and 
$\upset S = \{ b \in A \mid \exists a \in S. \ a \leq b \}$.  Let 
$\downset a = \downset \{ a \}$ and 
$\upset a = \upset \{ a \}$ 
for any $a \in A$.

\begin{definition}
For a frame $A$, a frame character $p:A \longrightarrow \twobb$ and an element $a \in A$, and a filter $F\subseteq A$, we use the following \emph{pre-point conditions}:
\begin{align}
\label{condition:pairs}
p(\algdiamond a) &= 0 \\
\label{condition:triples-lsc}
\{ c \in A \mid p(\algdiamond c) = 0 \} & \subseteq \downset a  \\
\label{condition:triples}
F_p & \subseteq F \\
\label{condition:triples-usc}
F & \subseteq F_p .
\end{align}
For each source and target category, in each version of the construction of a relational space from a modal frame, we display the relevant form of \emph{pre-points} and the applicable pre-point conditions in Table~\ref{table:pre-points}.
\end{definition}
\begin{table}[ht]
    \centering
    \begin{tabular}{|l|l|l|l|l|}
        \hline
       \makecell[lt]{Source\\Category\\ of\\Modal\\Frames,\\$\Ccal$} 
       & 
       \makecell[lt]{Target\\Category\\of\\Relational\\Spaces,\\$\Dcal$}  
       & 
       \makecell[lt]{Pre-point\\form} 
       & 
       \makecell[lt]{Pre-point\\ conditions}  
       &
       \makecell[lt]{Modal\\point\\conditions}
       \\ 
       \hline
       $\mfrm$ & $\relsp$  & $(p,a)$ & (\ref{condition:pairs})
       &
       (\ref{condition:pointsDiamond})
       \\
        $\mfrm$ & $\relspl$  & $(p,a)$ & (\ref{condition:pairs}),(\ref{condition:triples-lsc})
        & 
        (\ref{condition:pointsDiamond})
        \\
       $\mfrm$ & $\relspq$ & $(p,a,F)$ & (\ref{condition:pairs}),(\ref{condition:triples})
       &
       (\ref{condition:pointsDiamond}),
       (\ref{condition:pointsBox})
       \\
       $\mfrmds$ & $\relspql$ & $(p,a,F)$ & (\ref{condition:pairs}),(\ref{condition:triples-lsc}),(\ref{condition:triples})
       &
       (\ref{condition:pointsDiamond}),
       (\ref{condition:pointsBox})
       \\
       $\mfrmbs$ & $\relspqu$ & $(p,a,F)$ & 
       (\ref{condition:pairs}),(\ref{condition:triples}),(\ref{condition:triples-usc})
       &
       (\ref{condition:pointsDiamond}),
       (\ref{condition:pointsBox})
       \\
       $\mfrms$ & $\relspqc$ & $(p,a,F)$ & 
        (\ref{condition:pairs}),(\ref{condition:triples-lsc}),(\ref{condition:triples}),(\ref{condition:triples-usc})
        &
       (\ref{condition:pointsDiamond}),
       (\ref{condition:pointsBox})
        \\
       $\eqfrm$ & $\eqspq$ & $(p,a,F)$ 
       & 
       (\ref{condition:pairs}),(\ref{condition:triples-lsc}),(\ref{condition:triples}),(\ref{condition:triples-usc})
       &
       (\ref{condition:pointsDiamond}),
       (\ref{condition:pointsBox})
       \\ \hline
    \end{tabular}
    \caption{Pre-points}
    \label{table:pre-points}
\end{table}

The components of pre-points are, in some cases, fully-determined by the frame character component. In some cases, pre-points (and subsequent elements of the construction) could be done without mentioning the dependent components. We instead present pre-points as uniformly as possible to show exactly how the constructions in different cases relate to each other and exactly how the simplifications work. When the target category has only continuous p-morphisms, rather than pq-morphisms, the pre-points are not rather presented as triples satisfying both (\ref{condition:triples}) and (\ref{condition:triples-usc}), in particular fixing $F=F_p$: although it makes no difference at this stage, in the uniform development below it would turn out to require box-strictness of modal frame morphisms.
In order to allow us to write propositions regarding pre-points that are either triples or pairs, we use notation $(p,a,[F])$ to indicate that the third component should be ignored, where appropriate.

Conditions~(\ref{condition:pairs}) and (\ref{condition:triples-lsc}) are more obviously dual to each other if we state them, respectively, in the forms 
\begin{align*}
    \forall c \in A . \ & c \leq a \Longrightarrow p(\algdiamond c) = 0\\
    \forall c \in A . \ & p(\algdiamond c) = 0 \Longrightarrow c \leq a . \\
\end{align*}
Define a frame character $p$ of a modal frame to be \emph{replete} if 
\begin{equation*}
    p( \algdiamond a_p) = 0 . 
\end{equation*}
holds, and define a modal frame to be replete if every frame character is replete.
Conditions~(\ref{condition:pairs}) and (\ref{condition:triples-lsc}) do not, without further assumptions, guarantee that any  $p$ is replete.

Condition~(\ref{condition:pairs}) was used by Hilken \cite{hil00} alongside continuous p-morphisms. A related restriction is placed on theories constructed for the completeness results of Wijesekera \cite{wijesekera1990constructive} 
and in later developments \cite{AMPR01,FM97PLL,MendlerPaiva2005} using similar ideas. Condition~(\ref{condition:triples}) was introduced to allow for the use of continuous pq-morphisms. Condition~(\ref{condition:triples-lsc}) will be used to engineer lower-semicontinuity, while condition (\ref{condition:triples-usc}) will be used for upper-semicontinuity. 
Lemma~\ref{result:prepointProperties}  collects together simple consequences of these definitions for later use. 

\begin{lemma}
\label{result:prepointProperties}
Let $A$ be a modal frame. 
\begin{enumerate}
\item For each frame character $p$, there is always $a = \bot$ satisfying (\ref{condition:pairs}). If pre-points are not required to satisfy (\ref{condition:triples-lsc}) then $(p, \bot, [F_p])$ is a pre-point.

\item If pre-points are not required to satisfy (\ref{condition:triples-lsc}), and if $(p,a,[F])$ is a pre-point and $b \leq a$, then $(p,b,[F])$ is a pre-point.

\item   
If  $(p,a,[F])$ is a pre-point, then $a \leq a_p$. If $(p,a,F)$ satisfies (\ref{condition:triples-lsc}) and $p(\algdiamond a_p) = 0$, then $a_p \leq a$. 
If pre-points are required to satisfy both (\ref{condition:pairs}) and (\ref{condition:triples-lsc}), then all pre-points, if they exist, have the form $(p,a_p,[F])$  for some frame character $p$ [and filter $F$]. 


\item If $p$ is a frame character, then 
$F_p$
is a filter satisfying both (\ref{condition:triples}) and (\ref{condition:triples-usc}). If pre-points $(p,a,F)$ are required to satisfy both (\ref{condition:triples}) and (\ref{condition:triples-usc}) then $F=F_p$.

\item If pre-points are not required to satisfy  (\ref{condition:triples-usc}), and if $(p,a,F)$ is a pre-point and $F\subseteq G$ for a filter $G$, then $(p,a,G)$ is a pre-point.

\item If pre-points are not required to satisfy (\ref{condition:triples-usc}), and if $(p,a,F)$ and $(p,a,G)$ are pre-points, then so is $(p,a,F\cap G)$.
\end{enumerate}
\end{lemma}

We call pre-points of the form $(p,a_p,[F_p])$ \emph{canonical}, where they exist, that is, when $p$ is replete.

\subsubsection{Relating Pre-points}

 A binary relation $R_A$ is defined for each notion of pre-point under consideration. For frame characters $p$ and $q$, we consider the following \emph{relation conditions}
    \begin{align}
    \label{condition:relationTriplesBox}
        F
        & \subseteq \charop (q)
       \\ \label{condition:relationPairsDiamond}
        q(a) &= 0 
    \\
    \label{condition:relationPairsBox}
        p \circ \algbox &\leq q \\
\label{condition:relationPairsLSC}
        q &\leq p \circ \algdiamond .
    \end{align}
 
\begin{definition}
    For pre-points $(p,a)$ and $(q,b)$, define $(p,a) R_A (q,b)$ if and only if (\ref{condition:relationPairsDiamond}) and (\ref{condition:relationPairsBox})  both hold.
    
    For pre-points $(p,a,F)$ and $(q,b,G)$, define  $(p,a,F) R_A (q,b,G)$ if and only if (\ref{condition:relationPairsDiamond}) and (\ref{condition:relationTriplesBox}) both hold.
\end{definition}

In the cases where pre-points are  pairs, the definition is that used by Hilken \cite{hil00} and is, essentially, suggested by the definition given by Wijesekera \cite{wijesekera1990constructive} and later used by others \cite{AMPR01,MendlerPaiva2005}. 

The relation conditions (\ref{condition:relationPairsBox}) and (\ref{condition:relationPairsLSC}) express properties that hold of modal propositions in the usual relational Kripke semantics of classical modal logics. 
Conditions (\ref{condition:relationPairsBox}) and (\ref{condition:relationPairsLSC}) are, respectively, equivalent to 
\begin{align*}
    F_p &\subseteq \charop (q) \\
    q(a_p) &= 0 .
\end{align*}
    
Part~\ref{result:tripleRelationReduced} of Lemma~\ref{result:relationRelationships} shows that representation of pre-points as frame characters in the case where the target category for the construction is $\relspqc$ allows us to also to reconstruct the relation in the way familiar from the treatment from classical modal logic, except that in our case the two conditions (\ref{condition:relationPairsBox}) and (\ref{condition:relationPairsLSC}) are not equivalent to each other. 

\begin{lemma}
\label{result:relationRelationships}
\begin{enumerate}
    \item For a pre-point $(p,a,F)$ and any other frame character $q$, satisfaction of condition (\ref{condition:relationTriplesBox}) implies satisfaction of condition (\ref{condition:relationPairsBox}).
    
    \item For pre-points of the form $(p, a_p,[F])$, condition (\ref{condition:relationPairsDiamond}) is satisfied if and only if condition (\ref{condition:relationPairsLSC}) is satisfied.

    \item 
    \label{result:tripleRelationReduced}
    For canonical pre-points $(p,a_p,F_p)$ and any frame character $q$, 
    condition (\ref{condition:relationPairsDiamond}) is equivalent to condition  (\ref{condition:relationPairsLSC}), and 
    condition (\ref{condition:relationPairsBox}) is equivalent to condition (\ref{condition:relationTriplesBox}). 
    
    \end{enumerate}
\end{lemma}

%

We avoid giving different names to the different relations $R_A$ in each case of the construction. The obvious projections work: if $(p,a,F) R_A (q,b,G)$ in the case where the target is a relational space in $\relspq$, then $(p,a) R_A (q,b)$ for $\relsp$.

\subsubsection{Modal Frame Points and Topology}

There are usually too many pre-points,  if the intended destination is production of an adjoint to the functor $\Omega$. The set of pre-points is therefore pruned down to a set of points that are better behaved.

\begin{definition}
In each case of the construction, given the modal frame $A = (A , \algbox , \algdiamond )$, we define a set of \emph{(modal frame) points}, $
\Pbb _A$
This involves the application of the following \emph{modal point conditions} to all $(p,a,[F]) \in \Pbb _A$: 
\begin{align}
\label{condition:pointsDiamond}
c \nleq a 
& \Longrightarrow 
\exists (q,b,[G]) \in \Pbb_A \mbox{ with }
(p,a,[F]) R_A (q,b,[G]) 
\mbox{ and } q(c)=1\\ 
\label{condition:pointsBox}
c \notin F
& \Longrightarrow 
\exists (q,b,G) \in \Pbb_A 
\mbox{ with }
(p,A,F) R_A (q,b,G) 
\mbox{ and 
}
q(c)=0.
\end{align}
When the target category of relational spaces has continuous pq-morphisms as arrows, both (\ref{condition:pointsDiamond}) and (\ref{condition:pointsBox}) are required. When the target category has morphisms all the continuous p-morphisms between objects, only (\ref{condition:pointsDiamond}) is required.
This is summarized in Table~\ref{table:pre-points}. The set $\Pbb _A$ is the largest set of pre-points satisfying the applicable modal point conditions.
\end{definition}

Without further assumptions on $A$,  condition (\ref{condition:pointsBox}) is not automatically satisfied in the case $F=F_p$: when the target caregory has all continuous p-morphisms between the objects so that pre-points $(p,a)$ can be identified with pre-points $(p,a,F_p)$, the set of modal frame points would have been smaller if we were also to require (\ref{condition:pointsBox}). 

The following result shows how the modal point conditions simplify in the cases of a target category $\relspl$, $\relspqc$, or $\eqfrm$ where all pre-points can be assumed to be canonical, that is, of the form $(p,a_p,[F_p])$ for some frame character $p$. In these cases, every point simply consists of a replete frame character.

\begin{lemma}
    \begin{enumerate}
    \item For target categories $\relspl$, $\relspqc$ and $\eqfrm$, condition (\ref{condition:pointsDiamond}) is equivalent to: if 
    replete $p \in \Pbb_A$  and $p(\algdiamond c) = 1$, then there is replete $q \in \Pbb_A$ with $p R_A q$ and $q(c)=1$
    \item For target categories $\relspqc$ and $\eqfrm$, condition (\ref{condition:pointsBox}) is equivalent to: if replete $p \in \Pbb _A$ and $p(\algbox c) = 0$, then there is replete $q \in \Pbb_A$ with $p R_A q$ and $q(c)=0$.
\end{enumerate}
\end{lemma}

\begin{definition}
For each element $c$ in $A$, define the set 
\begin{equation*}
    \phi_A (c) = \{ ((p,a,F) \in \Pbb_A \mid p(c)=1 \} .
\end{equation*}

Define the topology $\Ofrak _A$ on $\Pbb _A$ to be given by the family of sets in the  image of $\phi_A: c \mapsto \phi_A (c)$. 
\end{definition}

The topology on $\Pbb_A$ is the weak/initial topology making the projection map $\pi: (p,a,[F]) \mapsto p \in \pt(A)$ continuous, and $\phi_A = \pi\invim \circ \eta_A$ where $\eta_A$ is the usual frame map from $A$ to the frame of opens on $\pt{(A)}$.

We now have enough data to present the object parts of the functors $\Fcal$ from categories of modal frames to categories of relational spaces.

\begin{definition}
For a modal frame $A$, let 
$\Fcal(A) = (\Pbb _A , \Ofrak_A , R_A)$ according to the case of the source category of modal frames and the target category of relational spaces (Table~\ref{table:pre-points}). 
\end{definition}

For a given frame character $p$, the component $a$ of a point $(p,a,[F])$ determines the interior of 
its set of $R_A$-unrelated points. The component $F$ determines the 
neighbourhood filter of the set of points $R_A$-related to $(p,a,F)$. 
This is the content of Lemma~\ref{tripmot}, which follows directly from the definition of pre-point, relation $R_A$, and the point conditions (\ref{condition:pointsDiamond}) and (\ref{condition:pointsBox}). 
\begin{lemma}
\label{tripmot}
Let $(p,a,[F]) \in \Pbb _A$. Then for all $c \in A$ 
\begin{enumerate}
\item $c \leq a$ iff $R_A \dirim (p,a,[F])\cap\phi _A (c) 
= \emptyset $ 
\item $c \in  F$ iff 
$R_A \dirim (p,a,F)
\subseteq \phi _A (c)$.
\end{enumerate}
\end{lemma}

The topology $\Ofrak _A$ can be equipped with modal operators as in Definition~\ref{def:OmegaAssignments}, giving a modal frame $\Omega (\Fcal (A))$. The point condition (\ref{condition:pointsDiamond}) is used to ensure that $\phi_A$ is a morphism of modal frames.

\begin{lemma}
Let $A$ be a modal frame. The function $\phi _A : A \longrightarrow \Omega (\Fcal (A))$ is a modal frame morphism. 
\end{lemma}

\begin{proof}
    That $\phi_A$ is a frame morphism is easily seen. 
    
    For lax preservation of the box operator, we require
    \begin{equation}
    \label{equation:phiBoxLaxPres}
        (\phi_A \circ \algbox )(c)\subseteq (\algbox \circ \phi _A)(c) 
    \end{equation}
    for all $c \in A$. 
    Suppse $(p,a,[F]) \in \phi_A ( \algbox c)$. Then $p(\algbox c)=1$ and $c \in F_p \subseteq F$. 
    If points are triples, then $F_p \subseteq F$ and the condition (\ref{condition:triples}) shows that $(p,a,F) \in \algboxclass (\phi_A(c))$. 
    If points are pairs, then condition (\ref{condition:relationPairsBox}) guarantees that $R_A \dirim (p,a) \subseteq \phi_A (c)$, and so  $(p,a) \in \algboxclass (\phi_A(c))$.
    Therefore $\phi_A (\algbox (c))\subseteq \algboxclass (\phi _A (c))$. As $\phi_A (\algbox (c))$ is open, (\ref{equation:phiBoxLaxPres}) holds.

    For lax preservation of the diamond operator, 
    we require
    \begin{equation}
    \label{equation:phiDiamondLaxPres}
        (\phi_A \circ \algdiamond)(c) \subseteq (\algdiamond \circ \phi _A )(c)
    \end{equation}
    for all $c \in A$. 
    Suppose $(p,a,[F]) \in \phi_A (\algdiamond c)$ for some $c \in A$. So $p(\algdiamond c)=1$ and $c\nleq a$. By the point condition (\ref{condition:pointsDiamond}), $(p,a,[F]) \in \algdiamondclass (\phi_A (c))$. Therefore $\phi_A (\algdiamond c) \subseteq \algdiamondclass (\phi_A (c))$. As 
    $\phi_A (\algdiamond c)$ is open, 
    $\phi_A ( \algdiamond c) \subseteq \algdiamond (\phi _A (c))$.
\end{proof}

These results can be strengthened if we take points to be triples and use stronger restrictions on pre-points. 

\begin{lemma}
\label{result:uscTripleImplyBoxStrict}
If all the modal frame points of $A$ satisfy conditions (\ref{condition:triples-usc}) and (\ref{condition:pointsBox}), then the morphism $\phi_A$ is box-strict and the space $\Fcal (A)$ is upper-semicontinuous. 
\end{lemma}

For every open set $\phi_A (c) \in \Ofrak_A$
\begin{equation*}
    \algbox \phi_A (c) = \algboxclass \phi_A (c) = \phi_A (\algbox c)
\end{equation*}

\begin{proof} 
We show 
$\algboxclass \circ \phi_A (c) \subseteq \phi_A \circ \algbox (c)$, having already established (\ref{equation:phiBoxLaxPres}).
Take $(p,a,F) \in \algboxclass (\phi_A (c))$ for some $c \in A$. Therefore $R_A \dirim (p,a,F) \subseteq \phi_A(c)$.  By Lemma~\ref{tripmot} we have that $c \in F$ since (\ref{condition:pointsBox}) is assumed. Then $p ( \algbox c) = 1$ by condition (\ref{condition:triples-usc}), and therefore $(p,a,F) \in \phi_A( \algbox c)$. 
\end{proof}

\begin{lemma}
\label{result:lscImpliesDiamoindStrict}
If the modal frame points of $A$ satisfy (\ref{condition:triples-lsc}), then the morphism $\phi_A$ is diamond-strict and the space $\Fcal(A)$ is lower semicontinuous.

For every open set $\phi_A (c) \in \Ofrak_A$:
\begin{equation*}
     \algdiamond \phi_A (c) =
     \algdiamondclass \phi_A (c) = \phi_A (\algdiamond c)
\end{equation*} 
\end{lemma}

\begin{proof}
 We show $\algdiamondclass \phi_A (c) \subseteq \phi _A ( \algdiamond c)$, having already established (\ref{equation:phiDiamondLaxPres}). 
 The condition (\ref{condition:triples-lsc}) guarantees that all points have the form $(p,a_p,[F])$ for some frame character $p$ (and filter $F$). 
    Suppose $(p,a_p,[F]) \in \algdiamondclass \phi _A (c)$ for some $c \in A$. So there is $(q,a_q,[G]) \in \Pbb_A$ with $(p,a_p,F) R_A (q,a_q,[G])$ and $q(c) =1$.  The condition (\ref{condition:relationPairsLSC}) holds because $q(a_p)=0$. Therefore  
    $p(\algdiamond c)=1$ and 
    $(p,a,F) \in \phi _A(\algdiamond c)$.
\end{proof}

Wikjesekera \cite{wijesekera1990constructive}, where $A$ is taken to be a distributive lattice with modalities and the relational space constructed is on the usual prime spectrum of $A$, produces a strict morphism from $A$ to the lattice of opens. In that case, points to be constructed to witness relation instances in the analogous proofs are constructed by direct appeals to the Prime Ideal Theorem. The distinction between points and pre-points is not required. The shadow of this remains in the present work in the treatment of modally spectral frames below.

Proposition~\ref{result:prepointEquivRel} shows that, in the case of $\eqfrm$ and $\eqspq$, the relation $R_A$ is an equivalence relation.
\begin{proposition}
\label{result:prepointEquivRel}
Suppose that points are triples satisfying both (\ref{condition:triples-lsc}) and (\ref{condition:triples-usc}):
    \begin{enumerate}
        \item If $A$ satisfies both (\ref{reflexivitybox}) and (\ref{reflexivitydiamond}) and then $R_A$ is reflexive
        \item If $A$ satisfies both (\ref{symmetrydiamondbox}) and (\ref{symmetryboxdiamond}) then $R_A$ is symmetric
         \item If $A$ satisfies both (\ref{transitivitybox}) and (\ref{transitivitydiamond}) then $R_A$ is transitive.
    \end{enumerate}
\end{proposition}

\begin{proof}
\begin{itemize}
(Reflexivity) We show that the conditions for $(p,a,F) R_A (p,a,F)$ are satisfied for any $(p,a,F) \in \Pbb_A$.  First, $p(a) \leq p(\algdiamond a) = 0$ using (\ref{reflexivitydiamond}). Let $c \in F$. By (\ref{condition:triples-usc}) and (\ref{reflexivitybox}), $1 = p(\algbox c) \leq p(c)$. 
Therefore 
$F \subseteq \{ c \in A \mid p(c)=1 \}$. 

(Symmetry)
    Suppose $(p,a,F)R_A(q,b,G)$. We have $p(b) \leq p(\algbox\algdiamond b) \leq  q (\algdiamond b) = 0$ since (\ref{condition:relationPairsBox}) holds and so does the box-diamond symmetry axiom (\ref{symmetryboxdiamond}). 

    For $(q,b,G)R_A (p,a,F)$, it remains to show that 
    $G \subseteq \{ c \in A \mid p(c) =1 \}$ as in (\ref{condition:relationTriplesBox}). 
    Let $c \in G$. 
    By upper-semicontinuity we have $q(\algbox c) =1$, and so 
    $(q,b,G) \in \phi_A( \algbox c)$. 
    By lower-semicontinuity of $R_A$, we have $(p,a,F) \in \algdiamond (\phi_A( \algbox c)) = \phi_A( \algdiamond \algbox c)$. Therefore $1=p(\algdiamond \algbox c) \leq p(c) $ by (\ref{symmetrydiamondbox}). 

(Transitivity)
    Suppose $(p,a,F) R_A (q,b,G)$ and $(q,b,G) R_A (r,c,H)$. 

    If $r(a)=1$, then diamond-strictness/lower-semicontinuity  applied twice gives 
    $(p,a,F) \in \phi_A (\algdiamond\algdiamond a)$. 
    We then have  
    $1= p(\algdiamond\algdiamond a) \leq p(\algdiamond a)$ using (\ref{transitivitydiamond}), which contradicts condition (\ref{condition:pairs}) for $(p,a,F)$ to be a pre-point. Therefore $r(a)=0$.

    For $(p,a,F)R_A (r,c,H)$, it remains to show that $F\subseteq \{ d \in A \mid r(d) = 1 \}$ as in (\ref{condition:relationTriplesBox}). 
    It suffices to show $F \subseteq G$. Let $d \in F$. Therefore
    $p(\algbox \algbox d) \geq p(\algbox d) =1$ 
    using (\ref{condition:triples-usc}) and (\ref{transitivitybox})
    As (\ref{condition:relationPairsBox}) holds by Proposition~\ref{result:relationRelationships}, we have $q( \algbox d)=1$, and so $d\in G$ by (\ref{condition:triples-usc}).

\end{itemize}
\end{proof}

The proof can be simplified using the canonical representation of pre-points as frame points and part 3 of Proposition~\ref{result:relationRelationships}, but this is perhaps less informative about how the conditions on pre-points, the relation $R_A$, and the equivalence frame axioms work together.
The reflexivity, symmetry and transitivity proofs are independent of one another and use different modal axioms. This is in the spirit of traditional correspondence theory for modal logics. The main difference is that only one axiom for each of the reflexivity, symmetry and transitivity properties is needed in the classical case: the two forms presented here are equivalent in classical modal logic. This mirrors  correspondence results in an alternative semantics of intuitionistic modal logics \cite{simpson94}.
It is possible to produce, without continuity assumptions, a similar result for reflexivity using a category of modal frames satisfying both (\ref{reflexivitybox}) and (\ref{reflexivitydiamond}) and a sub-category of $\relsp$ with only reflexive relations on spaces. 

\begin{proposition}
    Suppose $A$ is serial. If points are all triples $(p,a,F)$ satisfying (\ref{condition:triples-usc}), then 
    $R_A$ is serial. 
\end{proposition}

\begin{proof}
Since $p(\algdiamond a) = 0$, by seriality (\ref{ax:serial}) we have $p(\algbox a) = 0$ and so $a \notin F=F_p$. As $a \notin F$, by the point condition (\ref{condition:pointsBox}) there is a point $(q,b,G)$ with $(p,a,F) R_A (q,b, G)$. 
\end{proof}

A proof of seriality that applies in the cases of points as pairs and triples without the restriction (\ref{condition:triples-usc}) can be given in the case where $A$ is a modally spectral frame as discussed below.

\subsection{Adjunction}

\begin{theorem}
\label{result:adjunction}
There is a contravariant adjunction between each of the pairs of categories, consisting of a category of relational spaces and a category of modal frames, shown in Table~\ref{table:pre-points}. 
In each case, the assignment
$A \mapsto ( \mathbb{P} _A , \mathfrak{O} _A, \mathcal{R} _A )$ 
is the object part of the contravariant functor $\Fcal$  adjoint to 
$\Omega$. 
The unit on the modal frame side is given by $\phi _A$.
\end{theorem} 

Formally, some of the functors $\Omega$ were defined (see Table~\ref{table:OmegaFunctorsNEW}) to have codomains that are subcategories of the corresponding modal frame categories appearing in Table~\ref{table:pre-points}. The theorem should be read taking the necessary expansion of the codomains of $\Omega$. This occurs because the relevant versions of $\Fcal$ enforce semicontinuity conditions (as in Lemmas~\ref{result:uscTripleImplyBoxStrict} and \ref{result:lscImpliesDiamoindStrict}) on relational spaces, without requiring further structure on modal frames. Semicontinuity conditions are achieved by modifying $\Fcal$ to restrict to appropriate subspaces using conditions~(\ref{condition:triples-lsc}) and (\ref{condition:triples-usc}).

Take any relational space relational space $(X,\mathcal{O},R)$ and modal frame $A$. For each of the categories of relational spaces under consideration in Table~\ref{table:pre-points}, consider a morphism $f:A \longrightarrow \Omega (X, \Ocal , R)$ on the frame side.

\begin{definition} 
For each $x \in X$, define  
    \begin{align}
    p_x (b) = 1 &\mbox{ iff }  x \in f(b)\\
    a_x &= \bigvee \{ c \in A \mid R\dirim(x) \cap f(c) =\emptyset \} \\
    F_x &= \{ c \in A \mid R\dirim(x) \subseteq f(c) \}
    \end{align}
and, according to the relevant form of pre-point for each category of relational spaces and modal frames, define 
\begin{equation}
\label{definition:fhashTriple}
f^\hash (x) = (p_x , a_x, [F_x]) .
\end{equation}
\end{definition}

\begin{lemma}
The tuple $f^\hash (x)$ is a pre-point according to the appropriate conditions for each of the pairs of categories in Table~\ref{table:pre-points}.
\end{lemma}

\begin{proof}
    If $p_x (\algdiamond a_x) = 1$, then $x \in f(\algdiamond a_x) \subseteq \algdiamond f(a_x)$. Then there is some $y \in X$ such that $xRy$ and $y \in f(a_x)$. However,  $f(a_x) = \bigcup \{f(c) \mid R\dirim (f(c)) = \emptyset \}$. 
    Thus $p_x (\algdiamond a_x) = 0$ as required for the pre-point condition (\ref{condition:pairs}).

    Let $p_x(\algbox c)=1$. Then $x \in f(\algbox c) \subseteq \algbox f(c)$. For all $y$, if $xRy$ then $y \in f(c)$. Hence $c \in F_x$. 
    Therefore condition (\ref{condition:triples}) holds.

    Assuming $f$ to be diamond-strict and $(X,\Ocal, R)$ to be lower-semicontinuous, it is straightforward to verify that for any $c \in A$, that  
    $p_x (\algdiamond c) = 0$ implies $c \leq a_x$, so that condition (\ref{condition:triples-lsc}) is satisfied in this case.

    Assuming $f$ to be box-strict and $(X,\Ocal, R)$ to be upper-semicontinuous, we have that $c \in F_x$ iff $R\dirim (x) \subseteq f(c)$ iff $x \in \algbox f(c)$ iff $x \in f(\algbox c)$ iff $p_x (\algbox c) =1$. Therefore $F_x$ satisfies condition (\ref{condition:triples-usc}) in this case.
\end{proof}

For the cases where $f$ is diamond-strict and $(X, \Ocal ,R)$ is lower-semicontinuous, the component $a_x$ is the canonical element $a_{p_x}$ associated with $p_x$ as in (\ref{canonicalElement}). If $f$ is box-strict and $(X, \Ocal ,R)$ is upper-semicontinuous, then the component $F_x$ is the canonical filter $F_{p_x}$ associated with $p_x$ as in (\ref{canonicalFilter}). 

\begin{lemma}
\label{result:fHashRelPres}
If $xRy$ then $f^{\#}(x) R_A f^{\#}(y)$.
\end{lemma}

\begin{proof}
Suppose $xRy$. 
If $p_y(a_x)=1$ then 
$y \in f(a_x) = 
        \bigcup \{ f(c) \mid R\dirim (x) \cap f(c) = \emptyset \}$. 
This is impossible, so $p_y(a_x)=0$ and condition (\ref{condition:relationPairsDiamond}) is satisfied. 

If pre-points are triples, then to  demonstrate condition (\ref{condition:relationTriplesBox}), suppose $c \in F_x$. Since $xR y$, we have $y \in f(c)$ and 
hence $p_y(c)=1$. 

If pre-points are pairs, then condition (\ref{condition:relationPairsBox}) must be shown. If $p_x (\algbox c) =1$, then $x\in f(\algbox c) \subseteq \algbox f(c)$. As $xRy$, it follows that $y \in f(c)$, but then $p_y(c)=1$.
\end{proof}

\begin{lemma} 
\label{result:fHashGivesPoints}
For all $x \in X$, $f^{\#}(x)$ is a point. That is, $f^{\#}(x) \in \mathbb{P}_A$.
\end{lemma}

\begin{proof}
Let $P=\{ f^{\#} (x) \mid x \in X \}$ be the image of $f^{\#}$. We show that $P$ is closed under the conditions defining the set of points. 

For condition  (\ref{condition:pointsDiamond}), 
suppose $(p_x,a_x,[F_x]) \in P$ and $b \nleq a_x$. Then there is $y \in f(b)$ with $xRy$. 
By Lemma~\eqref{result:fHashRelPres}, 
$f^{\#}(x) R_A  f^{\#}(y)$, 
but we also have $f^{\#}(y) = (p_y,a_y,[F_y]) \in P$ and $p_y(b)=1$ as required.

For condition  (\ref{condition:pointsBox}), where points are triples, 
suppose that $(p_x,a_x,F_x) \in P$ and $c \notin F_x$. By definition, there is some $y \in X$ with $xRy \notin f(c)$.  Hence $f^{\#}(y) \in P$ has 
$f^{\#}(x) R_A f^{\#}(y)$ and $p_y(c)=0$.
\end{proof}

The above results show that 
\begin{equation*}
  f^\hash : (X , \Ocal , R) \longrightarrow ( \Pbb _A , \Ofrak _A , R_A) ; x \mapsto f^\hash (x) 
\end{equation*}
is a continuous relational morphism. 

\begin{lemma}
Suppose that $f$ is a morphism of $\mfrm$. Then the map $f^\hash$ is a continuous p-morphism. Moreover, for the cases where pre-points are triples, $f^\hash$ is a continuous pq-morphism.
\end{lemma}

\begin{proof}
    
    Suppose $f^\hash(x) R_A (q,b,[G]) \in \phi_A (c)$ for some $c \in A$. Then $q(a_x) = 0$ and $q(c)=1$. Since $c \nleq a_x$, there is some $y \in f(c)$ such that $xRy$. But then $f^\hash (x) R_A f^\hash (y) \in \phi _A (c)$. 

    For the proof of the q-morphism property, suppose $f^\hash(x) R_A (q,b,[G]) \notin \phi_A (c)$ for some $c \in A$. Since $q(c) = 0$, we have $c \notin F_x$ by definition of $R_A$.
    There is therefore some $y \in X$ with $xRy$ and $y \notin f(c)$. Then $f^{\#}(x) R_A f^{\#}(y)$ and $f^{\#} (y) \notin \phi _A (c)$.
\end{proof}

The following lemma makes essential use of the point condition (\ref{condition:pointsDiamond}), and also the point condition (\ref{condition:pointsBox}) in the cases using continuous pq-morphisms.

\begin{lemma}
\label{strmorphunique}
If $g:(X, \mathcal{O} ,\mathcal{R} ) \rightarrow 
( \mathbb{P}_A , \mathfrak{O} _A , \mathcal{R} _A )$ is any morphism of the relational space category from Table~\ref{table:pre-points}, such that $g$ satisfies 
$g^{\leftarrow} \circ \phi _A =f$, then $g=f^{\#}$.
\end{lemma}

\begin{proof}
We write the proof for points that are triples. For points that are pairs, simply elide the third, filter components of triples and the corresponding parts of the proof.
For all $x \in X$, let 
$g(x)= (q_x , b_x , G_x)$. 
We will show that 
$p_x = q_x$, $a_x =b_x $ and $F _x = G _x $.

Firstly, $p_x (c)=1$ iff $x \in f(c)$ iff 
$x \in g^{\leftarrow} ( \phi _A (c))$ 
iff $g(x) \in \phi _A (c)$ iff $q_x (c)=1$, so $p_x = q_x$.

If $a_x \nleq b_x$, then there 
is a $c \nleq b_x$ such that 
$R\dirim(x) \cap f(c) = \emptyset$. 
As $(q_x , b_x , G_x) \in \mathbb{P} _A$, there is $(q,b,G) \in \mathbb{P} 
_A$ such that $(q_x , b_x , G_x) R_A  (q,b,G)$ and $q(c)=1$, that is, $(q,b,G)  \in 
\phi _A (c)$. Since $g$ is a topological p-morphism, there is some 
$y$ such that 
$xR y$ and $g(y) \in \phi _A (c)$, and so $y \in f(c)$, which is a 
contradiction. Therefore $a_x \leq b_x$. 

If $b_x \nleq a_x$, then there is a $y \in f(b_x)$ such that 
$xR y$. Then $g(x) R_A g(y)$, so $p_y(b_x)=0$, and then   
$y \notin f(b_x)$, which is a contradiction. Therefore 
$b_x \leq a_x$. 

Suppose $F_x \not \subseteq G_x$. Then there is some $c \in F_x$ such that $c \notin G_x$. As $c \notin G_x$, there is 
$(q,b,G) \in \mathbb{P} _A$ with $(q_x , b_x , G_x) R_A (q,b,G)$ and $q(c)=0$, 
that is $(q,b,G) \notin \phi _A(c)$. Since $g$ satisfies the q-morphism 
property, there is some $y \in X$ with $xR y$ and $g(y) \notin \phi _A 
(c)$, that is $y \notin f(c)$. From $c \in F_x$ and $xRy$ it follows that $y \in f(c)$, a contradiction. 

Suppose $G_x \not \subseteq F_x$. Then there is a $c \in G_x$ with 
$c \notin F_x$. Since $c \notin F_x$, there is $y \in X$ with 
$xR y$ and $y \notin f(c)$. Then $g(x)R_A g(y)$, and so $G_x \subseteq \charop (p_y)$. As $c \in G_x$, we have $p_y(c)=1$, and so $g(y) \in \phi _A(c)$, and therefore and $y \in f(c)$, a contradiction. 
\end{proof}

Explicit presentations of the action of the point constructing functor on arrows and of the topological unit are given in the 
corollary below. 

\begin{corollary}
\label{result:adjointExplicit}
The contravariant functor 
$\Fcal$
adjoint to $\Omega$ is defined on modal frame morphisms 
$f:A \rightarrow B$ by 
\begin{equation}
\label{pointFunctorOnMorphs}
 \mathcal{F} (f)(q,b,[G]) = 
( q \circ f ,                           \bigvee \{ c \in A \mid f(c) \leq b \} ,  
[\{ c \in A \mid f(c) \in G \} ]). 
\end{equation}
For $X = (X, \Ocal , R)$, the unit 
\begin{equation*}
\psi _X :  X \longrightarrow 
( \Pbb _{\Omega X} , \Ofrak _{\Omega X} , R_{\Omega X}) ; 
x \mapsto 
( p_x , a_x , [F_x] ) 
\end{equation*}
on the relational frame side of the adjunction is defined by 
\begin{equation}
\label{psiXp}
p_x (U) = 1  \mbox{ iff }  x \in U 
\end{equation}
for each $U \in \Ocal$, and 
\begin{align} 
\label{psiXa}
a_x  & =  
\interiorop{(
X \setminus R\dirim(x))
}    \\
\label{psiXF}
F_x  & =  \{ U \in \Ocal \mid R\dirim (x) \subseteq U \} . 
\end{align}
\end{corollary}

The remainder of this section is an aside looking forward to potential applications of the theory, but playing no part in the rest of the present paper.

A \emph{lens} in a topological space is the intersection of a closed subset and saturated subset (an upper set in the specialization order). 
\begin{proposition}
 In the construction of a relation space $(\Pbb_A , \Ofrak_A , R_A)$ in $\relsp$ from a modal frame $A$, for each point $(p,a) \in \Pbb_A$, the image $R_A \dirim (p,a)$ is a closed set of $(\Pbb_A , \Ofrak_A)$.  

In the construction of a relation space $(\Pbb_A , \Ofrak_A , R_A)$ in $\relspq$ from a modal frame $A$, for each point $(p,a,F) \in \Pbb_A$, the image $R_A \dirim (p,a,F)$ is a lens  of $(\Pbb_A , \Ofrak_A)$. 
\end{proposition}
In a multi-modal and multi-relational setting, these properties enable the well-known lower (Hoare) and convex (topological Egli-Milner-Plotkin) powerdomain orders \cite{AbJ} (generalizing to avoid the requirement of compactness of lenses) to be employed to order the indexed family of relations constructed from an indexed family of pairs of modalities. 

\section{Duality}
\label{sec:duality}

The contravariant adjunction between $\topcat$ and $\frm$ restricts to a duality between the categories of sober spaces and spatial locales \cite{stonespaces}. 
We show how similar results work for each of the space/frame pairs of categories in Table~\ref{table:pre-points}, extending results from \cite{hil00}

\begin{definition}
For each category $\Ccal$ of modal frames in Table~\ref{table:pre-points}, the modal frame $A$ is \emph{$\Ccal$-spatial} if $\phi _A$ is an 
isomorphism.

For each category $\Dcal$ of relational spaces in Table~\ref{table:pre-points}, a relational space $X= (X,\Ocal ,R)$ is \emph{$\Dcal$-sober} if $\psi _X$ is an isomorphism.
\end{definition}

Let $\Ccal$, $\Dcal$ be a categories appearing as a pair in Table~\ref{table:pre-points}, where $\Ccal$ is the category of modal frames and $\Dcal$ is the matched category of relational spaces in the adjunction. 

\begin{theorem}
\label{pqdual}
There is a duality between the full subcategories of $\Ccal$ and 
$\Dcal$ whose objects lie in the images of the functors 
$\Omega      : \Dcal \longrightarrow \Ccal$ and
$\mathcal{F} : \Ccal \longrightarrow \Dcal$, respectively.
\end{theorem}

The theorem follows from Lemmas \ref{result:spatial} and \ref{result:sober}.

\begin{lemma}
\label{result:spatial}
If $X = (X , \Ocal , R)$ is a relational space considered as an object of $\Dcal$, then $A=\Omega (X)$ is $\Ccal$-spatial.
\end{lemma}

This result follows from $\Omega({\psi_X} ) \circ \phi _{A} = 1_{A}$, the triangle identity from the adjunction, and the surjectivity of $\phi _A$. It is exactly the same as in the non-modal case, except that it calls upon the relevant  adjunction (Theorem~\ref{result:adjunction}) for the pair $\Ccal$, $\Dcal$.


\begin{lemma}
\label{result:sober}
If $A$ is a frame in the category $\Ccal$, then $\mathcal{F} (A)$ is $\Dcal$-sober.
\end{lemma}

\begin{proof}
We give the proof in the case where points are triples. For the proof where points are pairs, elide the filter component of the triples.

We have 
$\mathcal{F}(\phi _A) \psi _{ \mathcal{F}(A)} = 1 _{\mathcal{F}(A)}$ 
by the triangle identities; we will show 
$\psi _{ \mathcal{F}(A) } \mathcal{F} (\phi _A) = 1 _{\mathcal{F}\Omega\mathcal{F}(A)}$.
Let $(p, U , F ) \in \mathcal{F} 
\Omega \mathcal{F} (A)$.

Let
\begin{equation*}
x = \mathcal{F}(\phi _A)(p, U, F ) 
 =  (p \circ \phi _A, e , K ) 
\end{equation*}
where 
\begin{align*}
    e &= \bigvee \{ c \in A | \phi _A (c) \subseteq U\} \\
    K &= \{ c \in A | \phi _A (c) \in F \} 
\end{align*}
as in (\ref{pointFunctorOnMorphs}).
Let $\psi _{\mathcal{F} A} (x) = (p_x , a_x , F_x)$ as in (\ref{psiXp}), (\ref{psiXa}), (\ref{psiXF}), where now $X = \Fcal(A)$.
We show $(p,U,F) = (p_x , a_x , F_x)$. 

For any $c \in A$, we have $p_x ( \phi _A (c))=1$ iff $x \in \phi _A(c)$ iff 
$p ( \phi _A (c))=1$, so $p_x = p$ as $\phi_A$ is surjective.

For any $c \in A$, it is the case that $\phi _A (c) \subseteq U$ iff 
$c \leq e$, by definition of $e$ and because $\phi_A$ is a frame morphism. 
However, 
$\phi _A(c) \subseteq a_x$ 
iff (by definition of $a_x$)
$\phi _A (c) \subseteq  X \setminus R_A\dirim (x)$ 
iff $c \leq e$, by definition of $x$ and Lemma~\ref{tripmot}.
Therefore $U = a_x$. 

For all $c \in A$ we have 
$\phi _A (c) \in F_x$ iff   
$R_A \dirim (x)  \subseteq \phi_A (c)$
iff 
$c \in K$ 
iff 
 $\phi _A (c) \in F$,                             
using Lemma~\ref{tripmot}. Therefore $F_x  = F$.
\end{proof}

\section{Existence of Points}
\label{sec:existence}

\subsection{Background}

Let $S$ be a subset of a frame $A$. The set $S$ said to be (upwards) \emph{directed} whenever any two elements of $S$ have an upper bound in $S$. A \emph{directed join} is a join of a directed set.

An element $a\in A$ is said to be \emph{compact} (or \emph{finite}) if, whenever $S$ is directed and $a \leq \bigvee S$, there is some $s \in S$ with $a \leq s$. The set of compact elements of $A$ is written $\compacts(A)$.

A frame $A$ is \emph{spectral} (or \emph{coherent}) if finite meets (including the empty meet, top element) of compacts, are compact, and every element of $A$ is the join of finite elements. In this case, $\compacts(A)$ is a distributive lattice.
A frame morphism $f: A \longrightarrow B$ is \emph{spectral} if it maps compact elements of $A$ to compact elements of $B$.

Let $\ideals(A)$ be the poset of ideals of a join-semilattice $A$, ordered by inclusion of ideals. The following results are standard. If $A$ is a distributive lattice, then $\ideals (A)$ is a frame. 
A frame $A$ is spectral if and only if there is a distributive lattice $B$ such that $A$ is isomorphic to the frame of ideals of $B$.  
If $A$ is a frame, then $B = \compacts (A)$, and the isomorphism is 
\begin{equation}
\label{def:mapToCompacts}
a \mapsto \compacts(A) \cap \downset a . 
\end{equation}If $B$ is a distributive lattice, then the compact elements of $\ideals(B)$ are the principal ideals. 
The category of distributive lattices is equivalent to the category of spectral frames. 
For a lattice $B$, the prime elements of $\ideals(B)$ are the prime ideals of $B$. 
Any spectral frame is spatial --- it has enough points. 
%

A subset $F$ of a partial order $A$ is said to be \emph{Scott-open}, if, whenever $S$ is a directed set and $\bigvee S \in F$, there is some $s \in S\cap F$.
Any Scott-open filter of a spectral frame $A$ corresponds to a filter of $K(A)$. 
If $F$ is a Scott-open filter of $A$ and $a \in \compacts(A)$, then 
$F \wedge a = \{ b\wedge a \mid b \in F \}$ is Scott-open.

If $F$ is a Scott-open filter, and $a\notin F$ is an element of $A$, then, using (\ref{def:mapToCompacts}) and the Prime Ideal Theorem, there is maximal and prime ideal $I$ containing $\compacts(A) \cap \downset a$ but disjoint from the filter $\{ \compacts (A) \cap \downset b \mid b \in F \}$.
There is therefore a completely prime filter of $A$ containing $F$ but not $a$.

\subsection{Modal Distributive Lattices and Spectral Frames}

\begin{definition}
    A \emph{modal distributive lattice} $(A,\algbox, \algdiamond)$ is a distributive lattice $A$ with a pair of monotone operators $\algbox, \algdiamond A  \longrightarrow A$ satisfying the axioms (\ref{m1}), (\ref{m2}), (\ref{m3}), (\ref{m4}).  
\end{definition}
We say $(A,\algbox, \algdiamond)$ is \emph{lower} if it satisfies axiom (\ref{axiom:diamondpreservesbinaryjoin}). It is \emph{convex} if it satisfies both (\ref{axiom:diamondpreservesbinaryjoin}) and (\ref{axiom:lscboxveemixed}).
It is \emph{serial} if it satisfies (\ref{ax:serial}).

\begin{definition}
A \emph{modal distributive lattice morphism} 
$f:(A, \algbox, \algdiamond) \longrightarrow (B, \algbox, \algdiamond)$ is a distributive lattice morphism $f:A \longrightarrow B$ such that (\ref{boxmorphcond}) and (\ref{dimorphcond}) both hold. It is \emph{box-strict} when (\ref{boxmorphcond}) is required to be an equality, \emph{diamond-strict} when (\ref{dimorphcond}) is required to be an equality and \emph{strict} when it is both box-strict and diamond-strict. 
Let $\mdlat$ be the category of modal distributive lattices and modal distributive lattice morphisms.
\end{definition}
The special case of a modal distributive lattice in which the lattice is a Heyting algebra is of interest in logical applications.

 Propositions~\ref{result:idealCompAdjunct}, \ref{result:distrib2spec} and \ref{result:spec2distrib} are already known \cite{hil00} and show that the ideal completion technique relating distributive lattices to frames extends to the modal setting.
\begin{proposition}
\label{result:idealCompAdjunct}
    The forgetful functor $\mfrm \longrightarrow \mdlat$ has a left adjoint
    $\Irs{}: \mdlat \longrightarrow \mfrm$ with 
    $\Irs(A) = (\ideals (A) , \algbox , \algdiamond)$, where
    \begin{align*}
        \algbox I &= \downset \{ \algbox a \mid a \in I \}\\
        \algdiamond I &= \downset \{ \algdiamond a \mid a \in I \}
    \end{align*}
    for $I \in \ideals (A)$. The unit of the adjunction is $a \mapsto \downset a$ and is strict.
\end{proposition}


\begin{definition}
A modal operator, $\algbox$ or $\algdiamond$, is said to be \emph{continuous} if it preserves directed joins. A modal operator is \emph{compact} if it takes compact elements to compact elements.
\end{definition}

\begin{definition}
A modal frame $(A, \algbox, \algdiamond)$ is \emph{modally  spectral} if it is spectral as a frame and the operators $\algbox$ and $\algdiamond$ are both compact and continuous.
\end{definition}

\begin{proposition}
\label{result:distrib2spec}
    If $B$ is a distributive lattice, then $\Irs(B)$ is modally spectral.
\end{proposition}

\begin{proposition}
\label{result:spec2distrib}
    If $A$ is modally spectral, then $a \mapsto  \compacts(A) \cap \downset a$  defines an isomorphism of modal frames with $\Irs(\compacts(A))$.
\end{proposition}
Note that an isomorphism of modal frames is necessarily strict, so this isomorphism lives in $\mfrms$.

The passage between distributive lattices and frames preserves many properties that are useful in correspondence theory.
\begin{lemma}
Let $A$ be modally spectral. Then for each of the inequalities 
(\ref{axiom:diamondpreservesbinaryjoin}),
(\ref{axiom:lscboxveemixed}), 
(\ref{ax:serial}), 
(\ref{reflexivitybox}),
(\ref{reflexivitydiamond}),
(\ref{transitivitybox}),
(\ref{transitivitydiamond}),
(\ref{symmetrydiamondbox}),
(\ref{symmetryboxdiamond}), 
$A$ satisfies the inequality for all elements if and only if $K(A)$ does.
\end{lemma}


Below, the parameter $\Ccal$ is used to run over categories of modal frames, and $\Dcal$ over ccategories of relational spaces, as in Table~\ref{table:pre-points}.
We will consider cases of the functor 
$\Fcal: \Ccal \longrightarrow \Dcal$ 
that constructs a relational space from a frame with operators, as described in Section~\ref{sec:adjunction}. 

\begin{lemma}
\label{result:relatedCharacterExists}
    Let $A$ be a modal frame. 
    If $A$ is modally spectral, $p(\algdiamond a) = 0$ and $p(\algdiamond b) = 1$, then there is a a frame character $q$ satisfying $q(a) = 0$, $q(b)=1$ and $p \circ \algbox \leq q$.
\end{lemma}

\begin{proof}
We know that the set $F_p$ is a filter by Lemma~\ref{result:FpFilter}. It is Scott-open, because $\algbox$ is continuous. 
Because $p(\algdiamond b)=1$ and $c$ is expressible as the join of compact elements below it, there is some compact element $c \leq b$ such that $p(\algdiamond c)=1$
The set $F_p \wedge c$ is Scott-open. Axiom (\ref{m3}) guarantees that $a \notin F_P \wedge c$. There is thus a completely prime filter $F$ containing $\upset(F_p \wedge c)$ but not $a$. The filter $F$ determines the required frame character $q$ via $F = \charop(q)$.
\end{proof}

\begin{lemma}
\label{result:pairprepoint2Point}
If $A$ is modally spectral, then for every frame character $p$ and element $a \in A$ such that $p(\algdiamond a)=0$, there exists $b\geq a$ such that $(p,b)$ is a modal frame point for $\Dcal = \relsp$.
Additionally, there is an $F$ such that $(p,b,F)$ is a point for $\Dcal = \relspq$.
\end{lemma}

\begin{proof}
For each $p$, $a$, $c$ with $p(\algdiamond a) = 0$ and $p(\algdiamond c)=1$, select a $q_{p, a, c}$ such that $q_{p, a, c} (a) = 0$, $q_{p, a, c}(c)=1$ and $p \circ \algbox \leq q_{p, a, c}$, noting that at least one exists by Lemma~\ref{result:relatedCharacterExists}.  

For each $p$, $a$ with $p(\algdiamond a) = 0$, define 
\begin{align*}
    S_{p,a} &= \{ d \in A \mid \forall c \in A. \ p( \algdiamond c)=1 \Rightarrow q_{p,a,c} (d)=0 \}\\ 
    b_{p,a} &= \bigvee S_{p,a} \\
    T_{p,a} &= \bigcap \{ \charop(q_{p,a,c}) \mid c\in A , p(\algdiamond c)=1 \} . 
\end{align*}

The set $S_{p,a}$ is easily verified to be directed, and continuity of $\algdiamond$ then ensures that $p(\algdiamond b_{p,a}) = 0$: if $p(\algdiamond d)=1$ then $q_{p,a,d}(d)=1$ and $d \notin S_{p,a}$. 
The set $T_{p,a}$ is a filter. 

We consider the set of pre-points constructed in this way. 
Define
\begin{equation*}
    P = \{ (p, b_{p,a}, [T_{p,a}] ) \mid a \in A, p:A \longrightarrow \twobb , p(\algdiamond a) = 0 \}
\end{equation*}
and consider $(p, b_{p,a}, [T_{p,a}] ) \in P$.

For any $p$, $a$ and $c$ with $p(\algdiamond a)=0$ and $p(\algdiamond c)=1$, the frame character $q_{p,a,c}$ is such that $q_{p,a,c} (b_{p,a}) = 0$ and $F_p \subseteq \charop (q_{p,a,c})$. Therefore, for any pre-point of the form $(q_{p,a,c},e,G)$, the relation $(p, b_{p,a}, [T_{p,a}] ) R_A (q_{p,a,c},e,[G])$ holds.

For any $p$, $a$ and $c$ with $p(\algdiamond a)=0$ and $p(\algdiamond c)=1$, there are $e=b_{q_{p,a,c},\bot}$ and $G=T_{q_{p,a,c},\bot}$ such that $(q_{p,a,c}, e,[G]) \in P$  because $q_{p,a,c} (\algdiamond \bot) = 0$. 

If $d\nleq b_{p,a}$, then $d \notin S_{p,a}$, so there is some $c$ such that $p(\algdiamond c)=1$  and $q_{p,a,c}(d)=1$. There is therefore $(q_{p,a,c},e,[G]) \in P$ such that the relation $(p, b_{p,a}, [T_{p,a}] ) R_A (q_{p,a,c},e,[G])$ holds, and so the set $P$ is closed under condition (\ref{condition:pointsDiamond}). 

If $d \notin T_{p,a}$, there is a frame character $c\in A$ with $p(\algdiamond c)=1$ and $q_{p,a,c}(d)=0$. This gives $(q_{p,a,c},e,G) \in P$ with $(p, b_{p,a}, T_{p,a} ) R_A (q_{p,a,c},e,G)$, and so closure under condition (\ref{condition:pointsBox}).
\end{proof}

Lemmas~\ref{result:relsplPoints} and \ref{result:canonicalTripleIsAPoint} show that the axioms making a modal frame lower or convex guarantee the existence of canonical modal frame points, and show that the relevant constructions (for lower-semicontinuous and convex spaces, respectively) simplify to constructions with modal points that are simply frame characters. In particular, all lower modal frames that are modally spectral are replete. 

\begin{lemma}
\label{result:relsplPoints}
    If $A$ is modally spectral and lower, then for every frame character $p$, 
    \begin{equation*}
        \label{equation:pdap0}
        p(\algdiamond a_p) = 0 .
    \end{equation*}
    If $\Fcal$ is the functor constructing objects of $\Dcal = \relspl$, then 
    $(p,a_p)$ is a point. Moreover, the set of points constructed from $A$ is in one-one correspondence with the set of frame characters: if $(p,a)$ is a pre-point then $a=a_p$. Representing points in this way, the relation $p R_A q$ holds if and only  (\ref{condition:relationPairsDiamond}) 
and 
(\ref{condition:relationPairsBox}) 
both hold.  
\end{lemma}

\begin{proof}
To see (\ref{equation:pdap0}), note that the operator $\algdiamond$ preserves all joins: it preserves directed joins because it is continuous, it preserves binary joins because is lower (\ref{axiom:diamondpreservesbinaryjoin}), and it preserves the empty join (\ref{m4}). If $(p,a)$ is a pre-point, then by definition of $a_p$, we have $a \leq a_p$, and by (\ref{condition:triples-lsc}), $a_p \leq a$.

We consider the set $P$ of all pairs $(q,a_q)$ where $q$ is any frame character. 
We need to check that the point condition (\ref{condition:pointsDiamond}) holds. Suppose $b \nleq a_p$. Then $p(\algdiamond a_p) = 0$ and $p (\algdiamond b) = 1$. Applying Lemma~\ref{result:relatedCharacterExists} gives $q$ such that $q(a_p) = 0$, $p \circ \algbox \leq q$ and $q(b)=1$. We then have $(p,a) R_A (q,a_q) \in P$. 

\end{proof}

\begin{lemma}
\label{result:canonicalTripleIsAPoint}
If $A$ is a convex modally spectral frame the functor $\Fcal$ is the one constructing objects in $\Dcal = \relspqc$, and $p$ is a frame character, then 
$(p,a_p,F_p)$ is a point (satisfying (\ref{condition:pairs}), (\ref{condition:triples}), (\ref{condition:triples-lsc}), (\ref{condition:triples-usc}), (\ref{condition:pointsDiamond}) and (\ref{condition:pointsBox})).

All points have the form $(p, a_p, F_p)$ for some frame character $p$, and for any other character $q$,
$p R_A q$ if and only if (\ref{condition:relationPairsDiamond}) and (\ref{condition:relationPairsBox}) both hold.
\end{lemma}

\begin{proof}
Since $A$ is lower, we have $p(\algdiamond a_p) = 0$, arguing as in Lemma~\ref{result:relsplPoints}. All pre-points $(p,a,F)$ satisfying (\ref{condition:pairs}) and condition (\ref{condition:triples-lsc}) must have $a=a_p$, and all satisfying (\ref{condition:triples}) and (\ref{condition:triples-usc}) must have $F = F_p$. 

Let 
\begin{equation*}
P = \{ (p,a_p,F_p) \mid p \mbox{ is a frame character of } A \} .
\end{equation*}
If $c\nleq a_p$ for any frame character $p$, then Lemma~\ref{result:relatedCharacterExists} gives the character $q$ required, so that $(p,a_p,F_p) R_A (q,a_q,F_q)$ and $q(c)=1$. Therefore $P$ is closed under (\ref{condition:pointsDiamond}).

Suppose $c \notin F_P$. 
Let $I = \{ a \in A \mid p(\algdiamond a) = 0 \}$. 
Then $I = \downset a_p$. 
We have 
$\algbox( c \vee a_p) \leq \algbox c \vee \algdiamond a_p$ using axiom (\ref{axiom:lscboxveemixed}) since the frame is convex. We therefore have $p(\algbox( c \vee a_p)) = 0$ as $p(\algbox c)=0$ and 
$p(\algdiamond a_p)=0$. 
So $c \vee a_p$ not contained in the Scott-open filter $F_p$. Therefore there is a completely prime filter $Q$ that contains $F_p$ but not $c \vee a_p$. The frame character $q = \charop(Q)$ has $q(c) = 0$, $q(a_p)=0$ and $F_p \subseteq Q$, so that $(p,a_p,F_p) R_A (q,a_q,F_q)$. Therefore $P$ is closed under (\ref{condition:pointsBox}). 
\end{proof}

Combining the results above, we see that the functor $\Fcal$ constructing lower-semicontinuous relational spaces becomes a functor that constructs continuous relational spaces upon restriction of the domain to convex modally spectral frames. Moreover, the functor maps morphisms between such frames to morphisms that satisfy the q-morphism property. 

\begin{lemma}
\label{result:convexImpliesContinuous}
    Suppose the point construction is $\Fcal: \lmfrmds \longrightarrow \relspl$. If $A$ is modally spectral and convex, then $\Fcal(A)$ is continuous. If $f:A \longrightarrow B$ is a morphism between convex modally spectral frames, then $\Fcal (f)$ is a continuous pq-morphism.
\end{lemma}


\begin{lemma}
\label{result:spatialSep}
Suppose that $A$ is modally spectral. 
Consider cases of the functor $\Fcal: \Ccal \longrightarrow \Dcal$ that constructs a relational space from a frame with operators as in Table~\ref{table:spatiality}. In these cases, the unit $\phi_A$ is one-one.
\begin{table}[ht]
\begin{center}
\begin{tabular}{|l|l|l|}
\hline
Spaces,  $\Dcal$    & 
Frames, $\Ccal$  &
Points \\ 
\hline
$\relsp$   & $\mfrm$ & $(p,a)$ \\
$\relspl$   & $\lmfrmds$ & $(p,a_p)$\\
$\relspq$  & $\mfrm$ & $(p,a,F)$ \\ 
$\relspqc$ & $\cmfrms$ & $(p,a_p,F_p)$\\
$\eqspqc$  & $\ceqfrm$ & $(p,a_p,F_p)$ \\
\hline
\end{tabular}
\caption{Cases making $\phi_A$ one-one}
\label{table:spatiality}
\end{center}
\end{table} 
\end{lemma}

\begin{proof}
We show that if $b \nleq a$ for any $a,b \in A$, then there is a point $(p,[a],[F]) \in \Pbb_A$ such that $p(b)=1$ and $p(a) =0$, and so $\phi_A (b) \nsubseteq \phi_A (a)$. 

In all cases from the table, spatiality of the frame $A$ gives a suitable frame character $p$.

    In the case $\Ccal = \mfrm$ and $\Dcal = \relsp$, the result follows from the fact that $(p,\bot$) is a pre-point (Lemma~\ref{result:prepointProperties}) and from the extension of $(p,\bot)$ to a point $(p,c)$ (Lemma~\ref{result:pairprepoint2Point}). 

In the case $\Ccal = \lmfrmds$, so that  $A$ is lower, and $\Dcal = \relspl$, by Lemma~\ref{result:relsplPoints} we have that  $(p,a_p)$ is a point. 

If $\Ccal = \mfrm$ and 
$\Dcal = \relspq$ then the result follows from  Lemma~\ref{result:pairprepoint2Point}.

If $\Ccal = \cmfrms$, so $A$ is convex, and $\Dcal = \relspqc$. By Lemma~\ref{result:canonicalTripleIsAPoint}, $(p,a_p,F_p)$ is a point as required. The special case of convex equivalence frames and continuous equivalence spaces follows from this.
\end{proof}

\begin{lemma}
    Let $A$ be modally spectral and from one of the categories $\Ccal$ from Table~\ref{table:spatiality}, and let $\Fcal$ be the matching instance of the modal point construction.  If $b\nleq \algdiamond a$, then $\phi_A(b) \nsubseteq \algdiamond \phi_A(a)$. Consequently, $\phi_A$ is diamond-strict. 
\end{lemma}

\begin{proof}
First, the diamond strictness consequence.   If $\phi_A (b) \subseteq \algdiamond \phi_A(a)$, then $b \leq \algdiamond a$, and so $\phi_A (b) \subseteq \phi_A (\algdiamond a)$. Since, for any $a$, it is the case that $\phi_A (b) = \algdiamond \phi_A(a)$ for some $b$, it is the also the case that $\algdiamond \phi_A (a) \subseteq \phi_A (\algdiamond a)$. 

For the initial claim: 
\begin{itemize}
    \item 
    Suppose $\Ccal = \mfrm$ and $\Dcal = \relsp$. Supposing $b \nleq \algdiamond a$, Lemma~\ref{result:spatialSep} gives a point $(p,d)$ such that $p(b) = 1$ and $p(\algdiamond a)=0$. 
 Lemma~\ref{result:pairprepoint2Point}  then gives a point $(p,c)$ with $a \leq c$. If $(p,c) R_A (q,d)$ for another point $(q,d)$, then $q(c) = 0$ and so $q(a)=0$. Therefore $(p,c) \notin \algdiamond (\phi_A (a))$ but $(p,c) \in \phi_A (b)$. 

    \item Suppose $\Ccal = \mfrm$ and $\Dcal = \relspq$. If $b\nleq \algdiamond a$, then by Lemma~\ref{result:spatialSep}, there is a point $(p,d,H)$ with $p(b)=1$ and $p(\algdiamond a)=0$. 
    By the final part of Lemma~\ref{result:pairprepoint2Point}, there is a filter $F$ and an element $c \geq a$ such that $(p,c,F)$ is a point. Then $(p,c,F)$ is in $\phi_A (b)$ but not $\phi_A (\algdiamond a)$, arguing as in the first case above.

    \item In cases where $\Fcal$ is such that points are all required to satisfy (\ref{condition:triples-lsc}), $\phi_A$ is already diamond-strict, by Lemma~\ref{result:lscImpliesDiamoindStrict}.
    Supposing $b\nleq \algdiamond a$, Lemma~\ref{result:spatialSep} combined with Lemma~\ref{result:relsplPoints} or Lemma~\ref{result:canonicalTripleIsAPoint}, as appropriate, gives a point $(p,a_p,[F_p])$ with $p(b) = 1$ and then $p(\algdiamond a)=0$. If it is the case that $(p,a_p,[F_p]) R_A (q,a_q,[F_q])$ for another point $(q,a_q,[F_q])$, then $q(a_p) = 0$ and so $q \leq p \circ \algdiamond$. Therefore $q(a)=0$, and so $(p,a_p,[F_p]) \notin \algdiamond (\phi_A (a))$ but $(p,a_p,[F_p]) \in \phi_A (b)$.
\end{itemize}
\end{proof}


\begin{lemma}
Let $\Fcal$ be one of the point constructions associated with Table~\ref{table:modalSpatiality}. 
Let $A$ be modally spectral and from one of the categories of frames in the table. 
If $b\nleq \algbox a$, then $\phi_A(b) \nsubseteq \algbox \phi_A(a)$.
As a consequence, $\phi_A$ is box-strict. 
\end{lemma}

\begin{proof}
First, the box-strictness. In the case of $\relspqc$ and $\cmfrms$, or $\eqspqc$ and $\ceqfrm$, we have already seen that $\phi_A$ is box-strict in Lemma~\ref{result:uscTripleImplyBoxStrict}. However, if, for any $a$, we have that $\phi_A(b) \subseteq \algbox \phi_A(a)$ implies $b \leq \algbox a$ for all $b$, then it also implies $\phi_A (b) \subseteq \phi_A (\algbox a)$. By taking 
$\phi_A (b) = \algbox \phi_A(a)$ we get 
$\algbox \phi_A (a) \subseteq \phi_A (\algbox a)$. 

Now suppose $b \nleq \algbox a$ and consider the various cases of the construction to show $\phi_A (b) \nsubseteq \algbox \phi_A(a)$.

    \begin{itemize}
        \item Take the case $\Ccal = \cmfrms$ and $\Dcal = \relspqc$. The case of $\ceqfrm$ and $\eqspqc$ is of course included in this. Suppose $b \nleq \algbox a$. By Lemma~\ref{result:spatialSep} there is a point $(p,a_p,F_p)$ such that $p(b) = 1$ and $p(\algbox a) =0$. We have $a \notin F_p$ so there is a point $(q,a_q,F_q)$ with $(p,a_p,F_p)R_A (q,a_q,F_q)$ and $q(a) = 0$, by (\ref{condition:pointsBox}). It follows that $(p,a_p,F_p)$ is in $\phi_A (b)$ but not $\algbox \phi_A (a)$.

        \item Suppose $\Ccal = \mfrm$ and $\Dcal = \relspq$.

        If $b \nleq \algbox a$, then by Lemma~\ref{result:spatialSep} there is a point $(p,c,F)$ with $p(b)=1$ and $p(\algbox a)=0$. We have $a \notin F_p$ and $F_p$ Scott-open. There is therefore a frame character $q_0$ such that 
        $F_p \subseteq \charop(q_0)$ and $q_0 (a) = 0$. It is then the case that $(q_0 , \bot , F_{q_0})$ is  a pre-point. From this follows the existence of a point $(q_0,e,G)$ by Lemma~\ref{result:pairprepoint2Point}.

        Let $c' = \bigvee \{ d \leq c \mid q_0(d)=0 \}$ and $F'  = F \cap \charop(q_0)$. 
        Now $p(\algdiamond c') \leq p(\algdiamond c) = 0$ and, moreover, $F'$ is a filter containing $F_p$. 
        Therefore the triple $(p,c',F')$ is a pre-point satisfying conditions (\ref{condition:pairs}) and (\ref{condition:triples}).
        
        We have $q_0 (c') = 0$ and $F' \subseteq \charop (q_0)$ and so 
        $(p,c',F') R_A (q_0,e,G)$. It
        then remains only to show that $(p,c',F')$ is a point. 
        
        Regarding the point condition (\ref{condition:pointsBox}), if $d \notin F'$, then $d \notin \charop(q_0)$ or  $d \notin F$. 
        If $d \notin \charop(q_0)$, then 
        $(q_0,e,G)$ provides the required witness.
        If $d \notin F$, then there is a point $(r,f,K)$ such that $(p,c,F) R_A (r,f,K)$ and $r(d) = 0$. 
        Then $r(c') \leq r(c) = 0$ and 
        $F' \subseteq F \subseteq \charop (r)$, so that $(p,c',F') R_A (r,f,K)$ and $(r,f,K)$ provides the witness.

        For the point condition (\ref{condition:pointsDiamond}), suppose $d \nleq c'$. Then $q_0 (d) = 1$ or $d \nleq c$. If $q_0 (d)=1$, then $(q_0,e,G)$ provides the required witness.
        If $d \nleq c$, then there is a point $(r,f,K)$ such that $(p,c,F) R_A (r,f,K)$ and $r(d) = 1$. 
        We then have $r(c')=0$ and $F' \subseteq \charop(r)$, so that 
        $(p,c',F') R_A (r,f,K)$ and $(r,f,K)$ provides the witness. 

        \item 
        The full proof in the case $\Ccal = \mfrm$ and $\Dcal = \relsp$ can be found in \cite{hil00}. The proof is essentially a slightly simplified version of the case above with $\Dcal = \relspq$, except the evident changes: (pre-)points are pairs, we don't need to check point condition (\ref{condition:pointsDiamond}) and, effectively, we are using $F=F'=F_p$ in the definition of $R_A$ for points $(p,c)$ and $(p,c')$ rather than, respectively, $(p,c,F)$ and $(p,c',F')$. 

    \end{itemize}
\end{proof}

\begin{theorem}
Consider the cases of $\Fcal: \Ccal \longrightarrow \Dcal$ for the pairs of categories shown in Table~\ref{table:modalSpatiality}. 
\begin{table}[ht]
\begin{center}
\begin{tabular}{|l|l|l|}
\hline
Space,  $\Dcal$    & 
Frames, $\Ccal$  &
Points \\ 
\hline
$\relsp$   & $\mfrm$ & $(p,a)$ \\
$\relspq$  & $\mfrm$ & $(p,a,F)$ \\ 
$\relspqc$ & $\cmfrms$ & $(p,a_p,F_p)$\\
$\eqspqc$  & $\ceqfrm$ & $(p,a_p,F_p)$ \\
\hline
\end{tabular}
\caption{Modally Spatial Situations}
\label{table:modalSpatiality}
\end{center}
\end{table} 
If $A$ is modally spectral, then $\phi_A$ is an isomorphism of $\Ccal$. 
\end{theorem}

\begin{proof}
The morphisms $\phi _A$ is a frame isomorphism since it is one-one and surjective. It is diamond-strict and box-strict. 
\end{proof}

Because of the coincidence recorded in Lemma~\ref{result:convexImpliesContinuous}, it is also the case that for the point construction $\Fcal : \lmfrmds \longrightarrow \relspl$, that if $A$ is both modally spectral and convex, then $\phi_A$ is box-strict and $A$ is modally spatial.

We recall that relations $R\subseteq X \times X$ such that all points $x \in X$ have images $R(x)$ tht are compact are an important case with many applications. The Vietoris space is normally constructed on compact subsets. If a relational space $(X,\Ocal,R)$ satisfies this restriction and $R$ is continuous, then all of the Vietoris axioms will apply, not merely (\ref{m1}), (\ref{m2}), (\ref{m3}), (\ref{m4}), (\ref{axiom:diamondpreservesbinaryjoin}), (\ref{axiom:lscboxveemixed}). It is therefore of interest to understand when the relation $R_A$ constructed from a modal frame $A$ has this property. Lemma~\ref{tripmot} gives the following characterization.
\begin{proposition}
For any modal frame point $(p,a,F)$ in the construction of a relational space in $\relspq$ from a modal frame in $\mfrm$, the image $R\dirim (p,a,F)$ is compact iff $F$ is Scott-open. 
\end{proposition}
In particular, for modally spectral frames, the set  
$R\dirim (p , a, F_p )$
is compact for all points 
$(p,a,F_p) \in \Pbb_A$. 
If we are using the construction of a continuous relational space, so that condition (\ref{condition:triples-usc}) holds of all points, then all points of the space have this property. By Lemma~\ref{result:convexImpliesContinuous}, this will also be the case if we are constructing a continuous relational space from a modally spectral convex frame, as would arise from ideal completion of a convex modal distributive lattice.

\bibliographystyle{plain}
\bibliography{refs}

\end{document}